\documentclass[12pt,reqno]{article}

\usepackage{xcolor}
\usepackage{amssymb}
\usepackage{amsthm}
\usepackage{enumitem}

\usepackage{stmaryrd}

\usepackage[colorlinks=true, linkcolor=webgreen, filecolor=webbrown, citecolor=webgreen]{hyperref}
\definecolor{webgreen}{rgb}{0,.5,0}

\usepackage{fullpage}
\usepackage{float}

\usepackage{mathtools}

\newcommand{\genlegendre}[4]{
  \genfrac{(}{)}{}{#1}{#3}{#4}
  \if\relax\detokenize{#2}\relax\else_{\!#2}\fi
}
\newcommand{\legendre}[3][]{\genlegendre{}{#1}{#2}{#3}}

\newcommand*{\carry}[1][1]{\overset{\ \textcolor{red}{#1}}}
\newcommand*{\Emptycarry}[1][1]{\overset{\ \textcolor{red}{#1}}{\textcolor{white}{1}}}

\newcommand{\floor}[1]{\lfloor #1 \rfloor}
\newcommand{\bfloor}[1]{\bigg\lfloor \displaystyle #1 \bigg\rfloor}
\newcommand{\ceil}[1]{\lceil #1 \rceil}

\makeatletter
\newcommand{\subalign}[1]{
  \vcenter{
    \Let@ \restore@math@cr \default@tag
    \baselineskip\fontdimen10 \scriptfont\tw@
    \advance\baselineskip\fontdimen12 \scriptfont\tw@
    \lineskip\thr@@\fontdimen8 \scriptfont\thr@@
    \lineskiplimit\lineskip
    \ialign{\hfil$\m@th\scriptstyle##$&$\m@th\scriptstyle{}##$\hfil\crcr
      #1\crcr
    }
  }
}
\makeatother

\begin{document}

\theoremstyle{plain}
\newtheorem{theorem}{Theorem}
\newtheorem{corollary}[theorem]{Corollary}
\newtheorem{lemma}[theorem]{Lemma}
\newtheorem{proposition}[theorem]{Proposition}

\theoremstyle{definition}
\newtheorem{conjecture}[theorem]{Conjecture}

\theoremstyle{remark}
\newtheorem{remark}[theorem]{Remark}

\begin{center}
\vskip 1cm{\LARGE\bf 
Multinomial Catalan Numbers and Lucas Analogues \\
}
\vskip 1cm
\large
Joaquim Cera Da Concei\c c\~ao\\
Normandie Universit\'e\\
UNICAEN, CNRS\\
Laboratoire de Math\'ematiques Nicolas Oresme\\
14000 Caen\\
France\\
\href{mailto:joaquim.cera-daconceicao@unicaen.fr}{\tt joaquim.cera-daconceicao@unicaen.fr} \\
\end{center}

\vskip .2in

\begin{abstract}
We define a new generalization of Catalan numbers to multinomial coefficients. With arithmetic methods, we study their integrality and the integrality of their Lucasnomial generalization. We give a complete characterization of regular Lucas sequences for which they yield integers up to finitely many cases.
\end{abstract}

\section{Introduction}\label{sec:intro}\label{Section1}

The central binomial coefficient $\binom{2n}{n}$ is well known to be divisible by $n+1$, thus forming the integer sequence of Catalan numbers $C(n)$. Indeed,
\[
C(n) = \frac{1}{n+1} \binom{2n}{n},
\]
for all integers $n\geq 0$. These numbers are of great interest since they appear in various contexts and have many combinatorial interpretations. Moreover, they have been generalized in several ways, such as the generalized Lucasnomial Fuss-Catalan numbers, defined by
\[
 \frac{U_r}{U_{(a-1)n+r}} \binom{an+r-1}{n}_U := \frac{U_r}{U_{(a-1)n+r}} \cdot\frac{U_{an+r-1} \cdots U_{(a-1)n+r}}{U_n \cdots U_1},
\]
for all integers $n\geq 1$, where $a\geq 2$ and $r\geq 1$ are fixed integers, $U=(U_n)_{n\geq 1}$ is a Lucas sequence and $\binom{an+r-1}{n}_U$ is a generalized binomial coefficient with respect to $U$, which we call a {\it Lucasnomial coefficient}. They are a surprising example of a generalization which yields integers only. This was shown by Ballot \cite{Ba1,Ba3} in recent papers which are of relevance to our work. A Lucas sequence $U$ is a second-order linear recurrence satisfying $U_0=0$, $U_1=1$ and $U_{n+2}=P U_{n+1}-QU_n$ for all $n\geq 0$, where $P,Q$ are integers. Generalized binomial coefficients with respect to an integer sequence $X=(X_n)_{n\geq 0}$, with $X_n\ne 0$ for all $n\geq 1$, are defined as follows:
\[
\binom{m}{n}_X =
\begin{cases}
    \frac{X_{m} X_{m-1} \cdots X_{m-n+1}}{X_n X_{n-1} \cdots X_1}, & \text{if $m\geq n \geq 1$;} \\
    1, & \text{if $n=0$;} \\
    0, & \text{otherwise.} \\
\end{cases}
\]

In this paper, we prove the integrality of all but a finite number of terms of yet another generalization of Catalan numbers, namely the {\it multinomial Catalan numbers}. These are defined for all integers $n\geq 1$ by
\[
C_l(n) =   \frac{1}{(n+1)^l} \binom{(l+1)n}{n}_l := \frac{1}{(n+1)^l} \binom{(l+1)n}{n,n, \ldots, n},
\]
where $l$ is a positive integer and 
\[
\binom{a_1+a_2+\cdots + a_{l+1}}{a_1,a_2, \ldots, a_{l+1}} = \frac{(a_1+\cdots + a_{l+1})!}{a_1! \cdots a_{l+1}!}
\]
are multinomial coefficients with positive integers $a_1, \ldots, a_{l+1}$. It is called a {\it central multinomial coefficient} when all the $a_i$'s are equal.  In particular, our study leads to the following generalization of the integrality of Catalan numbers:

\begin{theorem}\label{ThmIntro}
    The numbers $C_l(n)$ are integers for all $n\geq l$.
\end{theorem}

The usual Catalan numbers correspond to $l=1$. Proving Theorem \ref{ThmIntro} will require studying the $p$-adic valuation of multinomial coefficients, where $p$ represents a prime number. Kummer \cite{Ku} showed that in the binomial case, the $p$-adic valuation of $\binom{m+n}{n}$ is equal to the number of carries in the base-$p$ addition of $m$ and $n$. From Kummer's theorem, usually called Kummer's rule, and the identity
\begin{equation}\label{IdenIntro}
    \binom{a_1+a_2+\cdots + a_{l+1}}{a_1,a_2, \ldots, a_{l+1}} = \prod_{i=1}^{l+1} \binom{ a_1 + \cdots + a_i}{a_i},
\end{equation}
we obtain the following generalization of the Kummer rule:

\begin{theorem}[Generalized Kummer rule]\label{Kummersrule}
    The $p$-adic valuation of the multinomial coefficient
    \[
    \binom{a_1+a_2+\cdots+a_{l+1}}{a_1, a_2, \ldots, a_{l+1}},
    \]
    is equal to the number of carries that occur when adding $a_1, a_2, \ldots, a_l$, and $a_{l+1}$ in base $p$. The number of carries is independent of the order of the summands and of the parenthesization of the sum.
\end{theorem}

Note that the identity \eqref{IdenIntro} seems to indicate that we first add $a_1+a_2$, then $(a_1+a_2)+a_3$, etc., but the theorem states a stronger result. An expression $a_1 + a_2 + \cdots + a_{l+1}$ is said to be {\it parenthesized} when any pair of parentheses encloses exactly two summands. For instance, the expression $((a_1+a_2)+(a_3+a_4))$ is parenthesized, while $((a_1+a_2+a_3)+a_4)$ is not. The number of carries is {\it associative}, i.e., independent of the parenthesization of the sum $a_1+\cdots+a_{l+1}$. Indeed, given a parenthesization where the last addition to perform is $A+B$, where $A=a_1+\cdots+a_r$ and $B=a_{r+1}+\cdots +a_{l+1}$ for some $1\leq r\leq l$, we have
\[
M_{l+1}:=\binom{a_1+a_2+\cdots+a_{l+1}}{a_1, a_2, \ldots, a_{l+1}} =\binom{A}{a_1, \ldots, a_r}\binom{B }{a_{r+1}, \ldots, a_{l+1}} \binom{A+B}{B}.
\]
Thus, by an induction argument, one can write any multinomial coefficient as the product of binomial coefficients corresponding to each addition in the chosen parenthesization. Taking $p$-adic valuations in the final expression using the ordinary Kummer rule, we must find a total number of carries equal to $v_p(M_{l+1})$ for all parenthesizations. Moreover, for every permutation $\sigma \in S_{l+1}$, we have
\[
\binom{a_1+a_2+\cdots+a_{l+1}}{a_1, a_2, \ldots, a_{l+1}} = \binom{a_{\sigma(1)}+a_{\sigma(2)}+\cdots+a_{\sigma(l+1)}}{a_{\sigma(1)}, a_{\sigma(2)}, \ldots, a_{\sigma(l+1)}},
\]
hence the number of carries produced in the sum $a_{\sigma(1)}+\cdots +a_{\sigma(l+1)}$ is the same for all permutations $\sigma \in S_{l+1}$. Let us give an example of the need of parenthesizing. Adding four copies of $a=(111)_2$ in base $2$ in the parenthesized expression $((a+a)+(a+a))$  yields $9$ carries since we have
\begin{center}
    \begin{tabular}{ccccc}
          & $\Emptycarry$ & $\carry1$ & $\carry 1$ & $1$ \\
        + & & $1$ & $1$ & $1$ \\
        \hline
         & $1$ & $1$ & $1$ & $0$ \\
    \end{tabular}
    \quad and \quad 
    \begin{tabular}{cccccc}
          & $\Emptycarry$ & $\carry1$ & $\carry1$ & $1$ & $0$ \\
        + & & $1$ & $1$ & $1$ & $0$ \\
        \hline
         & $1$ & $1$ & $1$ & $0$ & $0$ \\
    \end{tabular},
\end{center}
and the first addition is performed twice. However, adding four copies of $a$ in the non-parenthesized expression $(a+a+a+a)$ yields only $3$ carries. Indeed, we have
\begin{center}
    \begin{tabular}{cccccc}
          & $\Emptycarry$ & $\Emptycarry$ & $\carry1$ & $1$ & $1$ \\
          & & & $1$ & $1$ & $1$ \\
          & & & $1$ & $1$ & $1$ \\
        + & & & $1$ & $1$ & $1$ \\
        \hline
          & $1$ & $1$ & $1$ & $0$ & $0$ \\
    \end{tabular}
\end{center}
since at each position, the addition of four $1$'s yields $(100)_2$, thus producing a carry two positions to the left. Therefore, it is important that additions be performed with parenthesization.

Section \ref{Section2} is dedicated to lemmas on the number of carries that occur when adding numbers of the form $p^j-1$ in base $p$, with $j\geq 1$, and to one property of carries in arbitrary base-$p$ additions.

In Section \ref{Section3}, we prove our main theorem on multinomial Catalan numbers, Theorem \ref{Thm3}, which states that the numbers $C_l(n)$ are non-integers for finitely many $n\geq 1$. Moreover, we give a necessary condition on the form of integers $n$ for which $C_l(n)$ is a non-integer.

We differ introducing several definitions, properties, and results concerning Lucas sequences and Lucasnomial coefficients to the beginning of Section \ref{Section4}. In particular, we will recall the primitive divisor theorem and the classification of $k$-defective regular Lucas sequences given by Bilu, Hanrot, and Voutier \cite{Bi}, and a version of Kummer's theorem suited for Lucasnomial coefficients \cite{Ba2}.

Section \ref{Section5} proves preliminary results on the $l${\it-Lucasnomial Catalan numbers}
\[
C_{U,l}(n) = \frac{1}{U_{n+1}^l} \binom{ (l+1)n}{n }_{U,l} := \frac{1}{U_{n+1}^l} \binom{ (l+1)n}{n,n, \ldots, n }_U,
\]
where $U$ is a Lucas sequence and $ \binom{ (l+1)n}{n}_{U,l}$ is a generalized multinomial coefficient with respect to $U$, which we call an $l${\it-Lucasnomial coefficient}. (They will be properly defined in Section \ref{Section4}.)

In Section \ref{Section6}, we proceed to find all pairs $(U,l)$ such that the $l$-Lucasnomial Catalan numbers $C_{U,l}(n)$ are integers for all but a finite number of integers $n\geq 1$. It turns out that only a few pairs $(U,l)$ yield only integers up to finitely many exceptions.

Throughout our paper, unless stated otherwise, the letters $l$ and $j$ stand for arbitrary fixed positive integers and the letter $p$ represents a prime number. We write $\mathbb{P}$ for the set of prime numbers. The $p$-adic valuation is denoted by $v_p$. For two integers $a\leq b$, we define $\llbracket a,b \rrbracket$ as the intersection $[a,b]\cap \mathbb{Z}$. The number of carries in a parenthesized base-$p$ addition of the positive integers $a_1,\ldots, a_{l+1}$ will be denoted by $C_p(l; a_1,\ldots, a_{l+1})$. When $a_1=a_2=\cdots =a_{l+1}$, the notation $C_p(l;a_1)$ will be used instead.

\section{Counting carries in base-\texorpdfstring{$p$}{TEXT} additions}\label{Section2}

In our paper and mostly in this section we use various notation for the $p$-ary expansion of an integer. For $N$ an integer and $a_\theta, a_{\theta-1}, \ldots, a_1, a_0$ its $p$-ary digits, the base-$p$ expansion of $N$ will sometimes be written as
\[
[a_\theta, a_{\theta-1}, \ldots, a_1, a_0]_p.
\]
Other times, it will be represented explicitly as the sum
\[
a_0+a_1 p + \cdots + a_\theta p^\theta = \sum_{i=0}^\theta a_i p^i.
\]
Given a positive integer $m$, we define the $m${\it-fold base-$p$ addition of $N$} as any parenthesized addition of $m$ copies of $N$ in base $p$. We proceed to prove two preliminary lemmas for the $m$-fold base-$p$ additions of numbers of the form $p^j-1$, $j\geq 1$. This section ends with a general lemma on the number of carries in base-$p$ additions. This lemma requires the use of the Legendre formula:
\[
v_p(n!)=\sum_{j\geq 1} \bfloor{\frac{n}{p^j}},
\]
valid for all $n\geq 0$.

\begin{lemma}\label{Lemma1}
    Let $m\in \llbracket 1, p \rrbracket$. The $m$-fold base-$p$ addition of $N=p^j-1$ generates $(m-1)j$ carries and the $p$-ary expansion of $mN$ is
    \[
    [(m-1), \underbrace{(p-1),\ldots,(p-1)}_{\text{$j-1$ times}}, (p-m)]_p.
    \]
\end{lemma}

\begin{proof}
    We prove the lemma by finite induction on $m$. For $m=1$, there is nothing to prove. Thus assume the statement holds for $m\in \llbracket1,p-1 \rrbracket$ and let us show that it is true for $m+1$. Let us choose a parenthesization so that the last base-$p$ addition we perform is $mN+N$. By the induction hypothesis, $(m-1)j$ carries are produced to obtain $mN$ and
    \[
    mN=[(m-1), \underbrace{(p-1),\ldots,(p-1)}_{\text{$j-1$ times}}, (p-m)]_p.
    \]
    As $p-m\geq 1$, adding the first digits of $mN$ and $N$ yields $(p-m)+(p-1) = p+(p-m-1)$. Therefore, the first digit of $(m+1)N$ is $p-m-1$ and one initial carry is produced. Moreover, the initial carry propagates and generates $j-1$ further carries. As the first $j$ digits of $N$ are $p-1$, the digits of $(m+1)N$, from the second to the $j$-th, are $p-1$. The last carry is added to the sum of the $(j+1)$-st digits of $mN$ and $N$ and does not generate another carry as $(m-1)+1= m<p$. Hence
    \[
    (m+1)N=[m, \underbrace{(p-1),\ldots,(p-1)}_{\text{$j-1$ times}}, (p-m-1)]_p,
    \]
    and the addition did produce $(m-1)j+j=mj$ carries.
\end{proof}

\begin{remark}\label{rem1}
    A similar result holds if we replace $p^j-1$ by $p^s(p^j-1)$ for some $s\geq 1$, since this will amount to shifting all operations $s$ digits to the left, leaving $s$ zero digits to the right.
\end{remark}

\begin{lemma}\label{Lemma2}
    For all integers $i\geq 0$, a $p^i$-fold base-$p$ addition of $p^j-1$ produces $(p^i-1)j$ carries.
\end{lemma}

\begin{proof}
    Write $N=p^j-1$. For $i=0$, there is nothing to prove. Assume the statement is true for some $i\geq 0$ and let us show it is true for $i+1$. First, write
    \begin{equation}\label{p^ifoldSum}
        p^{i+1}N = \sum_{k=1}^{p^i} pN.
    \end{equation}
    By Lemma \ref{Lemma1}, the $p$-fold base-$p$ addition of $N$ produces $(p-1)j$ carries, and this is done a total of $p^i$ times on the right hand side of \eqref{p^ifoldSum}. Thus, $p^i(p-1)j$ carries are produced in this way. By Remark \ref{rem1}, the number of carries occurring in the $p^i$-fold addition of $pN$ is the same as the number of carries produced in the $p^i$-fold base-$p$ addition of $N$. By the induction hypothesis, this addition produces $(p^i-1)j$ carries. Therefore, the total number of carries counted is
    \[
    p^i(p-1)j + (p^i-1)j = (p^{i+1}-1)j,
    \]
    and this completes the proof.
\end{proof}

\begin{theorem}\label{ThmCarries}
    Let $m\in \llbracket 1, p^j-1\rrbracket$ and write $m=p^\nu M$ with $p\nmid M$. The $m$-fold base-$p$ addition of $N=p^j-1$ generates $(m-1)j$ carries and the $p$-ary expansion of $MN$ is
    \[
    \big[m_\theta,\ldots, m_1,(m_0-1), \underbrace{(p-1),\ldots,(p-1)}_{\text{$j-\theta-1$ times}}, (p-m_\theta-1),\ldots,(p-m_1-1),(p-m_0) \big]_p,
    \]
    where $M=[m_\theta,\ldots, m_1,m_0]_p$ with $0\leq \theta< j$ and $m_\theta\geq 1$.
\end{theorem}

\begin{proof}
    First, we assume $\nu =0$, i.e., $m=M$, and we proceed by induction of $\theta$. If $\theta=0$, then $M\in \llbracket1, p-1\rrbracket$ and the result follows from Lemma \ref{Lemma1}. As the result holds for $j=1$, we may assume $j\geq 2$. Assume the result holds for all $t\leq \theta$ with $\theta \in \llbracket 0, j-2 \rrbracket$ and let us show it is true for $\theta+1$. Let $M'=M-m_{\theta+1}p^{\theta+1}$, and write $M'=[m_t, \ldots, m_1,m_0]_p$, where $t\leq \theta$. In the addition
    \[
    \biggl( \sum_{i=1}^{M'} N \biggr) + \biggl(\sum_{i=1}^{m_{\theta+1}p^{\theta+1}} N \biggr),
    \]
    we let $S_1$ and $S_2$ represent respectively the first and the second sums above. Since $t\leq \theta$, exactly $(M'-1)j$ carries are produced to obtain $S_1$ by the induction hypothesis. Moreover, since
    \[
    S_2=\sum_{i=1}^{m_{\theta+1}} p^{\theta+1} N,
    \]
    using Lemma \ref{Lemma2}, followed by Lemma \ref{Lemma1} and Remark \ref{rem1}, exactly
    \[
    m_{\theta+1}(p^{\theta+1}-1)j + (m_{\theta+1}-1)j = (m_{\theta+1}p^{\theta+1}-1)j
    \]
    carries are produced to obtain $S_2$. Thus, so far we have
    \[
    (M'-1)j+(m_{\theta+1}p^{\theta+1}-1)j = (M-2)j
    \]
    carries. It remains to show that $j$ carries occur in the base-$p$ addition $S_1+S_2$. Let us give the $p$-ary expansions of $S_1$ and $S_2$. By the induction hypothesis, $S_1$ is
    \[
    [ m_t,\ldots, m_1,(m_0-1), \underbrace{(p-1),\ldots,(p-1)}_{\text{$j-t-1$ times}}, (p-m_t-1),\ldots,(p-m_1-1),(p-m_0) \big ]_p.
    \]
    For $S_2$, we find that
    \[
    S_2=[(m_{\theta+1}-1), \underbrace{(p-1),\ldots,(p-1)}_{\text{$j-1$ times}}, (p-m_{\theta+1}),\underbrace{0,\ldots,0}_{\text{$\theta+1$ times}}]_p,
    \]
    by Lemma \ref{Lemma1} and Remark \ref{rem1}. The first $\theta+1$ digits of $S_1+S_2$ are those of $S_1$ and no carry occurs. Next, adding the $(\theta+2)$-nd digits together, we have
    \[
    (p-1)+(p-m_{\theta+1}) = p + (p-m_{\theta+1}-1),
    \]
    so that the $(\theta+2)$-nd digit of $S_1+S_2$ is $(p-m_{\theta+1}-1)$ and a carry occurred. Since the $j-1$ following digits of $S_2$ are all equal to $p-1$, the first carry generates $j-1$ further carries. Moreover, this implies that the $j-1$ digits of $S_1+S_2$ starting from position $\theta+3$ are those of $S_1$. Since $S_1$ has exactly $t+j+1$ digits, there is no addition left to consider and the $p$-ary expansion of $S_1+S_2$ is
    \[
    \big[m_{\theta+1},\ldots, m_1,(m_0-1), \underbrace{(p-1),\ldots,(p-1)}_{\text{$j-\theta-2$ times}}, (p-m_{\theta+1}-1),\ldots,(p-m_1-1),(p-m_0) \big]_p.
    \]
    Therefore, we counted exactly $j$ carries in the addition $S_1+S_2$, making for a total of $(M-1)j$ carries. This ends the proof for $\nu =0$. If $\nu\geq 1$, then the $p$-ary expansion of $mN=p^\nu MN$ is a $\nu$-left shift of the digits of $MN$ leaving $\nu$ zero digits to the right. As for the number of carries, note that
    \[
    mN = \sum_{i=1}^M p^\nu N,
    \]
    where $(M-1)j$ carries are produced in the $M$-fold base-$p$ addition of $p^\nu N$ by Remark \ref{rem1} and the case $\nu =0$. Moreover, the $p^\nu$-fold base-$p$ addition of $N$ produces exactly $(p^\nu-1)j$ carries by Lemma \ref{Lemma2}, so that a total of $(M-1)j+M(p^\nu -1)j =(m-1)j$ carries are produced in the $m$-fold base-$p$ addition of $N$.
\end{proof}

Theorem \ref{ThmCarries} and Lemma \ref{Lemma2} show a definite pattern in the number of carries of an $m$-fold base-$p$ addition of $N=p^j-1$ for all $m\leq p^j$, but this pattern breaks up at $m=p^j+1$. Indeed, we have
\[
(m-1)N = p^jN = [\underbrace{(p-1),\ldots,(p-1)}_{\text{$j$ times}}, \underbrace{0,\ldots,0}_{\text{$j$ times}}]_p
\]
by Theorem \ref{ThmCarries}. Therefore, adding $N$ to $(m-1)N$ produces no carry since $N$ has exactly $j$ digits. Thus the total number of carries in the $m$-fold base-$p$ addition of $N$ is equal to the number of carries in its $(m-1)$-fold base-$p$ addition, that is, $(m-2)j$ carries.

\begin{lemma}\label{carriesSum}
    Let $a_1,\ldots, a_{l+1}$ be positive integers and write $a_i=A_i p^j + B_i$ with integers $B_i<p^j$ and $A_i\geq 0$ for all $1\leq i\leq l+1$. Then
    \[
    C_p(l;a_1,\ldots, a_{l+1}) \geq C_p(l; A_1, \ldots, A_{l+1}) +C_p(l; B_1, \ldots, B_{l+1}).
    \]
\end{lemma}

\begin{proof}
    By Legendre's formula, we have
    \begin{equation}\label{legendreApp1}
        v_p\bigg(\binom{a_1 + \cdots +a_{l+1}}{a_1, \ldots, a_{l+1}}\bigg) = \sum_{i\geq 1} \biggl( \bfloor{\frac{a_1+\cdots + a_{l+1}}{p^i}} -\sum_{k=1}^{l+1} \bfloor{\frac{a_k}{p^i}} \biggr).
    \end{equation}
    When $i\leq j$, we have
    \begin{equation}\label{floor1}
        \bfloor{\frac{a_k}{p^i}} = A_kp^{j-i} + \bfloor{\frac{B_k}{p^i}},
    \end{equation}
    and
    \begin{equation}\label{floor2}
        \bfloor{\frac{a_1+\cdots + a_{l+1}}{p^i}} =\sum_{k=1}^{l+1} A_k p^{j-i}  +  \bfloor{\frac{B_1+\cdots + B_{l+1}}{p^i}}.
    \end{equation}
    When $i>j$, using the identity $\floor{x+y}=\floor{x}+\floor{y}+\floor{\{x\}+\{y\}}$ valid for all $x,y\in \mathbb{R}$, we obtain
    \[
    \bfloor{\frac{a_k}{p^i}} = \bfloor{\frac{A_k}{p^{i-j}}}+\bfloor{\frac{B_k}{p^i}} +\bfloor{ \bigg\{ \frac{A_k}{p^{i-j}} \bigg\} + \bigg\{ \frac{B_k}{p^i} \bigg\} }.
    \]
    Let $A_k=[d_\theta,\ldots, d_1,d_0]_p$. Then
    \[
    \bigg\{ \frac{A_k}{p^{i-j}} \bigg\} = \bigg\{ \frac{d_0+\cdots+d_{i-j-1}p^{i-j-1}}{p^{i-j}} \bigg\} \leq 1-\frac{1}{p^{i-j}},
    \]
    since $d_i\leq p-1$ for all $0\leq i \leq \theta$. With the same method, we obtain
    \[
     \bigg\{ \frac{B_k}{p^{i}} \bigg\} \leq \frac{1}{p^{i-j}} - \frac{1}{p^i}.
    \]
    Hence
    \begin{equation}\label{floor3}
        \bfloor{\frac{a_k}{p^i}} =\bfloor{\frac{A_k}{p^{i-j}}}+\bfloor{\frac{B_k}{p^i}}= \bfloor{\frac{A_k}{p^{i-j}}}.
    \end{equation}
    With the assumption $i>j$, and using again the identity $\floor{x+y}=\floor{x}+\floor{y}+\floor{\{x\}+\{y\}}$, we have
    \begin{equation}\label{floor4}
        \bfloor{\frac{a_1+\cdots + a_{l+1}}{p^i}} \geq \bfloor{\frac{A_1 +\cdots + A_{l+1} }{p^{i-j}}} + \bfloor{\frac{B_1+\cdots + B_{l+1}}{p^i}}.
    \end{equation}
    Putting together  \eqref{legendreApp1}, \eqref{floor1}, \eqref{floor2}, \eqref{floor3} and \eqref{floor4}, we find that the $p$-adic valuation of $\binom{a_1 + \cdots +a_{l+1}}{a_1, \ldots, a_{l+1}}$ is at least equal to
    \begin{align*}
        S:=&\sum_{1\leq i \leq j} \biggl( \bfloor{\frac{B_1+\cdots + B_{l+1}}{p^i}} - \sum_{k=1}^{l+1} \bfloor{\frac{B_k}{p^i}} \biggr) \\
        &+ \sum_{i>j} \biggl( \bfloor{\frac{A_1 +\cdots + A_{l+1} }{p^{i-j}}} +\bfloor{\frac{B_1+\cdots + B_{l+1}}{p^i}} - \sum_{k=1}^{l+1}  \biggl(\bfloor{\frac{A_k}{p^{i-j}}}+\bfloor{\frac{B_k}{p^i}} \biggr) \biggr).
    \end{align*}
    But, by Legendre's formula, $S$ is equal to
    \[
    v_p\biggl(\binom{A_1 + \cdots +A_{l+1}}{A_1, \ldots, A_{l+1}}\biggr)+v_p\biggl(\binom{B_1 + \cdots +B_{l+1}}{B_1, \ldots, B_{l+1}}\biggr),
    \]
    and the result follows from Kummer's theorem.
\end{proof}

Finally, note that to prove these results and the facts given in the introduction, the primality of $p$ is needed. Indeed, Legendre's formula and Kummer's theorem have been widely used, and they both require that $p$ be a prime.

\section{Integrality of multinomial Catalan numbers}\label{Section3}

Our proof of the integrality of multinomial Catalan numbers amounts to checking that the $p$-adic valuation of $C_l(n)$ is non-negative for all primes $p$. We end up proving that $C_l(n)$ is not an integer for only a finite number of integers $n\geq 1$. For these finitely many exceptional numbers, we give a necessary condition they must satisfy. The proof has two main steps. In the theorem below, we first show that $C_l(n)$ has a negative $p$-adic valuation only when $p\leq p^{v_p(n+1)}\leq l$. Then, we prove that for all pairs $(n,p)$ such that $p\leq p^{v_p(n+1)}\leq l$ and $v_p(C_l(n))<0$, the integer $n+1$ must be a power of $p$. In particular, this shows that any exceptional number $n$ satisfies $n+1 \leq l$, thus ensuring they are finitely many.

\begin{theorem}\label{Thm1}
    The number $C_l(n)$ is an integer if $v_p(n+1)\geq \log_p{(l+1)}$ for all primes $p$ such that $p\leq l$ and $p\mid n+1$.
\end{theorem}

\begin{proof}
    Since multinomial coefficients are integers, $v_p(C_l(n))\geq 0$ for every $p\nmid n+1$. Hence, assume $p\mid n+1$ and let $j=v_p(n+1)$. It suffices to show that the $p$-adic valuation of the central multinomial coefficient is greater than or equal to $lj$. Let $\lambda\geq 1$ be the unique integer such that $n+1=\lambda p^j$. Then
    \[ 
    n= (\lambda-1)p^j + \sum_{i=0}^{j-1} (p-1)p^i := (\lambda-1)p^j + N.
    \]
    By Kummer's rule, the valuation of the multinomial coefficient is obtained by counting the number of carries in the $(l+1)$-fold base-$p$ addition of $n$. On the one hand, if $p\geq l+1$ then Lemma \ref{Lemma1} ensures that exactly $lj$ carries are produced in the $(l+1)$-fold base-$p$ addition of $N$. Hence, by Lemma \ref{carriesSum}, there are at least $lj$ carries in the $(l+1)$-fold base-$p$ addition of $n$. On the other hand, if $p\leq l$ and $j\geq \log_p{(l+1)}$ then $l\leq p^j-1$, and Theorem \ref{ThmCarries} and Lemma \ref{Lemma2} ensure that at least $lj$ carries are produced in the $(l+1)$-fold base-$p$ addition of $n$.
\end{proof}

Note that Theorem \ref{Thm1} is enough to recover the integrality of the usual Catalan numbers since there is no prime less than $l=1$. However, for higher values of $l$ there are still infinitely many $n\geq 1$ to study. For instance, when $l=2$, the integrality of $C_l(n)$ for positive integers $n\equiv 1$ (mod $4$) is not settled by Theorem \ref{Thm1}.

\begin{lemma}\label{Lemma4}
    Suppose $C_l(n)$ is not an integer. Then $n+1$ is a prime power and $l+1$ has at least two non-zero digits in base $n+1$.
\end{lemma}

\begin{proof}
    Since $C_l(n)$ is not an integer, by Theorem \ref{Thm1}, there exists a prime number $p\leq l$ such that $v_p(C_l(n))<0$ and $n+1=\lambda p^j$ for some integers $1\leq j <\log_p{(l+1)} $ and $\lambda\geq 1$, $p\nmid \lambda$. By contradiction, assume that $\lambda\geq 2$. First write
    \[
    n=(\lambda-1)p^j + \sum_{i=0}^{j-1} (p-1)p^i :=(\lambda-1)p^j +N.
    \]
    Since $l+1> p^j$, we may write
    \[
    l+1 = \sum_{k=\nu}^\theta a_k p^{kj},
    \]
    in base $p^j$, where $\theta\geq \nu \geq 0$ and $0\leq a_k<p^j$ with $a_\nu, a_\theta>0$. In the decomposition
    \[
    \sum_{i=1}^{l+1} N = \sum_{ \subalign{k=\nu \\ a_k\ne 0} }^\theta \sum_{i=1}^{a_k p^{kj}} N =\sum_{ \subalign{k=\nu \\ a_k\ne 0} }^\theta \sum_{i=1}^{a_k } \sum_{s=1}^{p^{kj}} N,
    \]
    we count the carries produced in the two inner sums using Lemma \ref{Lemma2} and Theorem \ref{ThmCarries}, and obtain
    \begin{equation}\label{carries}
        \sum_{ \subalign{k=\nu \\ a_k\ne 0} }^\theta \left( a_k(p^{kj}-1)j+(a_k-1)j \right) = \sum_{ \subalign{k=\nu \\ a_k\ne 0} }^\theta (a_kp^{kj}-1)j = (l+1)j-\eta j
    \end{equation}
    carries, where $\eta$ is the number of non-zero $p^j$-ary digits of $l+1$. Note that if $l+1$ had only one non-zero $p^j$-ary digit, then we would have $lj$ carries and $C_l(n)$ would have a non-negative $p$-adic valuation, which is a contradiction. Hence $l+1$ has at least two non-zero digits in base $p^j$. Since $\lambda\geq 2$, we know that $(\lambda-1)p^j$ has at least one non-zero $p$-ary digit, say $b$. Therefore, at least one carry occurs when adding $\ceil{p/b}$ copies of $(\lambda-1)p^j$ together. By euclidean division, write $l+1=Qp^j+R$, where $0\leq R <p^j$ and $Q\geq 1$. Then, there are at least $\floor{(l+1)/p}=Qp^{j-1}+\floor{R/p}$ carries produced in the $(l+1)$-fold base-$p$ addition of $(\lambda-1)p^j$. In particular, $Q$ is either equal to
    \[
    \sum_{k=\nu}^\theta a_kp^{(k-1)j} \quad \text{or to} \quad \sum_{k=1}^\theta a_k p^{(k-1)j},
    \]
    depending on whether $\nu >0$ or $\nu = 0$ respectively. In both cases, since $l+1$ has at least two non-zero digits in base $p^j$, we have $Q\geq p^{(\theta-1)j}$ and $Qp^{j-1} \geq p^{\theta j-1} \geq \theta j$ for every $\theta, j\geq 1$. This, together with \eqref{carries} and Lemma \ref{carriesSum}, yields that at least
    \[
    (l+1)j -\eta j +Qp^{j-1} \geq lj + (\theta-\eta+1)j \geq lj
    \]
    carries are produced in the $(l+1)$-fold base-$p$ addition of $n$. Indeed, $\theta+1$ is the number of base-$p^j$ digits of $l+1$. Thus, $\theta+1\geq \eta$, the number of non-zero base-$p^j$ digits of $l+1$. This means that $v_p(C_l(n))\geq 0$, which is a contradiction. Hence $\lambda=1$, $n+1=p^j$, and $l+1$ has at least two non-zero digits in base $n+1$.
\end{proof}

\begin{theorem}\label{Thm3}
    There are at most finitely many integers $n\geq 1$ such that $C_l(n)$ is not an integer. Such integers must be of the form
    \[
    n=p^j-1,
    \]
    with $p^j\leq l$ and $l\ne a p^{kj}-1$ for every $k\geq 1$ and $a\in \llbracket 1, p-1 \rrbracket$.

\end{theorem}

\begin{proof}
    By Theorem \ref{Thm1} and Lemma \ref{Lemma4}, if there exists $n\geq 1$ such that $C_l(n)$ is not an integer then $n=p^j-1$, with $p\leq l$ and $1\leq j<\log_p{(l+1)}$, that is $p^j\leq l$. Moreover, Lemma \ref{Lemma4} ensures that $l+1$ has at least two $p^j$-ary digits. In other words, $l+1$ is not of the form $ap^{kj}$ with $k\geq 1$ and $a\in \llbracket 1, p-1 \rrbracket$.
\end{proof}

\begin{remark}
    These integers are not necessarily counterexamples. Indeed, if $l=3$ and $n=1$, then $n$ has the form given in Theorem \ref{Thm3}, but $C_3(1)=3$.
\end{remark}

\begin{corollary}
    The multinomial Catalan numbers $C_l(n)$ are integers for all $n\geq l$.
\end{corollary}

This last corollary is a direct generalization of the integrality of the usual Catalan numbers, which we may recover by taking $l=1$. Also note that if $l$ is a prime, then $C_l(l-1)$ is not an integer and $l-1$ is then the largest counterexample.

\section{Preliminaries to the Lucasnomial case}\label{Section4}

We highlight some useful properties and results concerning Lucas sequences. Most of them can be found in the work of Ballot \cite{Ba1,Ba2,Ba3} or in the book of Williams \cite{Wi}. Recall that a Lucas sequence $U=U(P,Q)=(U_n)_{n\geq 0}$ is a second-order linear recurrence with initial conditions $U_0=0$ and $U_1=1$, which satisfies
\[
U_{n+2}=PU_{n+1}-QU_{n},
\]
for all $n\geq 0$, where $P$ and $Q$ are non-zero integers. Also, we define $\Delta=P^2-4Q$ the discriminant of $U(P,Q)$, that is, the discriminant of the characteristic polynomial of $U$. Given the definition, all terms of a Lucas sequence are integers. Moreover, a Lucas sequence has many divisibility properties. In particular, $U$ is a {\it divisibility sequence}, i.e., for all $m,n\geq 0$, we have
\[
m\mid n \implies U_m \mid U_n.
\]
We say that $U$ is {\it non-degenerate} if $U_n\ne 0$ for all $n\geq 1$. A way of checking the non-degeneracy of $U$ is to show that $U_{12}\ne 0$, which is a sufficient condition. If $U$ is non-degenerate and $\gcd{(P,Q)}=1$, we say that $U$ is a {\it regular} Lucas sequence. Such a sequence is a strong divisibility sequence, i.e., it satisfies
\[
\gcd{(U_m,U_n)}=|U_{\gcd{(m,n)}}|,
\]
for all $m,n\geq 0$. If $U$ is regular and $\Delta=0$, we find that $U$ is either the sequence of non-negative integers or defined by $U_n=(-1)^{n-1}n$. Throughout the paper, we assume that $U$ is a $\Delta${\it -regular Lucas sequence}, that is, a regular Lucas sequence such that $\Delta\ne 0$. The second part of our paper is dedicated to studying the divisibility of the $l$-Lucasnomial coefficient
\[
\binom{(l+1)n}{n, \ldots, n}_U := \binom{(l+1)n}{n}_{U,l},
\]
by $U_{n+1}^l$, for all $n\geq 0$. Let $n!_U$ represent the {\it generalized factorial} with respect to $U$, i.e., $n!_U= U_n U_{n-1}\cdots  U_1$ for all $n\geq 1$ and $0!_U=1$. Here, $l$-Lucasnomial coefficients are defined by the formula
\[
\binom{a_1+\cdots +a_{l+1}}{a_1,\ldots,a_{l+1}}_U := \frac{(a_1+\cdots+a_{l+1})!_U}{a_1!_U \cdots a_{l+1}!_U},
\]
where the $a_i$'s are non-negative integers. Moreover, we make the assumption that $P>0$ because
\[
U(-P,Q)=(-1)^{n-1} U(P,Q),
\]
for all $n\geq 0$, implies that the $l$-Lucasnomial coefficients with respect to $U(P,Q)$ and $U(-P,Q)$ are the same up to their sign.

We start by stating and proving different facts about Lucas sequences and $l$-Lucasnomials. These include prime divisibility properties of $U$. We define the rank of appearance $\rho=\rho_U(p)$ of a prime $p$ in $U$ to be the least positive integer $n$ such that $p\mid U_n$. The condition $p\nmid Q$ guaranties the existence of $\rho$. Furthermore, the regularity of $U$ implies that $p$ divides some term of $U$ of positive index if and only if $p\nmid Q$. From now on, $p$ represents a prime not dividing $Q$ of rank $\rho$. The following proposition lists interesting and useful properties of the rank:

\begin{proposition}\label{lawsAandR}
    Let $U(P,Q)$ be a non-degenerate Lucas sequence. Then, for all positive integers $m$ and $n$, we have
    \begin{enumerate}
        \item $\rho \mid p-\legendre{\Delta}{p}$ if $p>2$, $\rho_U(2) =2$ if $2\mid \Delta$ and $\rho_U(2)=3$ if $2\nmid \Delta$;

        \item $p\mid U_n$ if and only if $\rho \mid n$;

        \item $v_p(U_{\rho m}) = v_p(U_\rho)+x + \delta_x$,
    \end{enumerate}
    where $\legendre{*}{*}$ is the Legendre symbol, $x=v_p(m)$,
    \[
    \delta := v_p\biggl(\frac{P^2-3Q}{2}\biggr)\cdot [p=2] \cdot [2\nmid P],
    \]
    and $\delta_x=\delta \cdot [x>0]$.
\end{proposition}

Let $\mathcal{P}$ be a statement. In Proposition \ref{lawsAandR}, we used the notation $[\mathcal{P}]$. It is called the {\it Iverson symbol} and is defined by
\[
[\mathcal{P}] = 
\begin{cases}
    1, & \text{if $\mathcal{P}$ is true;} \\
    0, & \text{if not.}
\end{cases}
\]

In order to decide the integrality of $l$-Lucasnomial Catalan numbers, we approach the problem $p$-adically just as we did for the usual multinomial Catalan numbers. We proceed to state a Lucasnomial version of Kummer's theorem \cite{Ba2}.

\begin{theorem}[Kummer's rule for Lucasnomials]
    Let $U(P,Q)$ be a non-degenerate Lucas sequence and $m$ and $n$ be positive integers. Then, the $p$-adic valuation of the Lucasnomial coefficient $\binom{m+n}{n}_U$ is equal to $v_p(U_\rho)$ when the proposition
    \[
    \mathcal{P} : \bigg\{ \frac{m}{\rho} \bigg\} + \bigg\{ \frac{n}{\rho} \bigg\} \geq 1,
    \]
    is true, plus the number of carries in the base-$p$ parenthesized addition 
    \[
    \bfloor{ \frac{m}{\rho} } + \bfloor{\frac{n}{\rho} } +[\mathcal{P}],
    \]
    plus $\delta$ when a carry occurs from the first to the second digit, where $\delta$ is the integer defined in Proposition \ref{lawsAandR}.
\end{theorem}

Note that this result can be generalized to other kinds of sequences \cite{KW} using a method similar to the classical Kummer case. For $l$-Lucasnomials, we will not use a Lucasnomial version of Theorem \ref{Kummersrule}, but only the following corollary:

\begin{corollary}\label{TLucMultiVal}
    The $p$-adic valuation of the $l$-Lucasnomial coefficient $\binom{ a_1+\cdots + a_{l+1}}{a_1, \ldots, a_{l+1} }_U$ is equal to the sum
    \[
    \sum_{i=1}^{l} v_p\biggl( \binom{ a_1+\cdots + a_{i+1}}{a_{i+1} }_U\biggr),
    \]
    where the $a_i$'s are positive integers.
\end{corollary}

\begin{proof}
    It suffices to note that for every $l\geq 2$, we have
    \begin{align*}
        \binom{ a_1+\cdots + a_{l+1}}{a_1, \ldots, a_{l+1} }_{U} &= \frac{(a_1+\cdots +a_{l+1})!_U}{a_1!_U \cdots a_{l+1}!_U} \\
        &=  \frac{(a_1+\cdots +a_{l+1})!_U}{a_{l+1}!_U \cdot (a_1+\cdots+ a_l)!_U} \cdot \frac{(a_1+\cdots+ a_l)!_U}{a_1!_U \cdots a_l!_U } \\
        &= \binom{ a_1+\cdots + a_{l+1}}{a_{l+1}}_{U} \cdot \binom{ a_1+\cdots + a_l}{a_1, \ldots, a_l }_{U}.
    \end{align*}
    Applying this recurrence relation $(l-1)$-times, we obtain
    \[
    \binom{ a_1+\cdots + a_{l+1}}{a_1, \ldots, a_{l+1} }_{U} = \prod_{i=1}^l \binom{a_1+\cdots + a_{i+1}}{a_{i+1}}_U,
    \]
    and the result follows by taking $p$-adic valuations.
\end{proof}

Throughout the rest of the paper, given $U$ and $l$, the sequence $C_{U,l}=\big(C_{U,l}(n)\big)_{n\geq 1}$ is called a {\it multi-Catalan sequence} if for all but finitely many integers $n\geq 1$, the numbers
\[
C_{U,l}(n) = \frac{1}{U_{n+1}^l} \binom{(l+1)n}{n}_{U,l}
\]
are integers. As stated briefly in the introduction, only a few pairs $(U,l)$ satisfy this condition. To find infinite families of integers $n\geq 1$ such that $C_{U,l}(n)$ does not belong to $\mathbb{Z}$, we mostly rely on a result of Bilu, Hanrot, and Voutier \cite{Bi} on $\Delta$-regular Lucas sequences. For $k\geq1$ an integer, we say that $U$ is $k$-{\it defective} if $U_k$ has no primitive prime divisor, i.e., there is no prime $p\nmid \Delta$ of rank $k$ in $U$. They proved the following theorem:

\begin{theorem}[Primitive divisor theorem]
    The sequence $U$ is not $n$-defective for all $n>30$.
\end{theorem}

In addition to this strong result, all $n$-defective $\Delta$-regular Lucas sequences have been found for every positive integer $n\leq 30$, we list them in Table \ref{TableA}. We proceed to prove two preliminary lemmas that determine the $p$-adic valuation of some Lucasnomial coefficients.

\begin{lemma}\label{valuationlucasnomials}
    Let $n=\lambda \rho p^j-1$, $p\nmid \lambda$, $j\geq 0$, and $i\in \llbracket 1,\rho \rrbracket$. Then
    \[
    v_p\biggl(\binom{ (i+1)n}{n }_U \biggr) = j+\delta_j+v_p(U_\rho)\cdot [i<\rho] + v_p(\Lambda_i) \cdot [j>0 \text{ or } i<\rho] +c_i,
    \]
    where $c_i$ is the number of carries produced in the base-$p$ addition of $(\lambda-1)$ and $(\lambda i-1)$, and $\Lambda_i=\lambda (i+1)-1$.
\end{lemma}

\begin{proof}
    First, write
    \[
    \frac{n}{\rho} = \lambda p^j -1 + \frac{\rho-1}{\rho} := (\lambda-1)p^j + N + \frac{\rho-1}{\rho},
    \]
    where $N=\textstyle \sum_{k=0}^{j-1} (p-1)p^k$. Note that
    \[
    \frac{in}{\rho} = (\lambda i-1)p^j + N + \frac{\rho-i}{\rho},
    \]
    for every $1\leq i \leq \rho$. Let us apply Kummer's rule for Lucasnomials. When adding the fractional parts, we obtain
    \[
    \bigg\{ \frac{n}{\rho} \bigg\} + \bigg\{ \frac{in}{\rho} \bigg\} = \frac{2\rho -i-1}{\rho} =
    \begin{cases}
        1 + \frac{\rho-i-1}{\rho}, & \text{if $1\leq i <\rho$;} \\
        \frac{\rho-1}{\rho}, & \text{if $i=\rho$.}
    \end{cases}
    \]
    Thus, when $i<\rho$, a carry with weight $v_p(U_\rho)$ occurs across the radix point. Given the digits of $N$, this carry generates at least $j$ other carries with total weight $j+\delta_j$. When $i=\rho$, no carry occurs across the radix point, but if $j>0$ then the first digits of the integer parts, $\floor{n/\rho}$ and $\floor{in/\rho}$, being $p-1$ and $p-1$, we obtain one carry with weight $1+\delta_j$, which generates at least $j-1$ more carries. Therefore, we are left with counting carries produced in the parenthesized base-$p$ addition
    \[
    (\lambda i-1) p^j + (\lambda-1)p^j + [j>0 \text{ or } i<\rho] \cdot p^j,
    \]
    or equivalently, counting carries in the addition $(\lambda i-1) p^j + (\lambda-1)p^j$, plus the number of carries in the addition $(\lambda(i+1)-2)p^j + p^j$ when $i<\rho$ or $j>0$. The latter produces exactly $v_p(\Lambda_i)$ carries. 
\end{proof}

\begin{lemma}\label{valuationlucasnomials2}
    Let $n=\lambda \rho p^j-1$, $p\nmid \lambda$, $j\geq 0$, and $i\in \llbracket \rho+1,2\rho\rrbracket$. Put $ \Lambda_{i,j}=\lambda(i+1)-1-[j=0]$. If $\rho < i < 2\rho$, then
    \[
    v_p\biggl(\binom{(i+1)n}{n}_U\biggr) =j+\delta_j +v_p(U_\rho)+v_p(\Lambda_{i,j})+c_i,
    \]
    with $c_i$ the number of carries produced in the base-$p$ addition of $(\lambda-1)$ and $(\lambda i-1-[j=0])$. If $i=2\rho$, then
    \[
    v_p\biggl(\binom{ (i+1)n}{n }_U \biggr)=c_i+
    \begin{cases}
        j+  v_p(\Lambda_{i,j}) \cdot  [j>0], & \text{if $p>2$;}\\
        j+v_2(\Lambda_{i,j})-1, & \text{if $p=2$ and $j>1$;} \\
        0, & \text{if $p=2$ and $j\in \{0,1\}$,}
    \end{cases}
    \]
    with $c_i$ the number of carries produced in the base-$p$ addition of $(\lambda-1)$ and $(\lambda i-1-[j=0])$.
\end{lemma}

\begin{proof}
    Write $i=\rho+t$, with $1\leq t \leq \rho$, so that with the notation used in the proof of Lemma \ref{valuationlucasnomials}, we have
    \[
    \frac{n}{\rho} = (\lambda-1)p^j+N+\frac{\rho-1}{\rho},
    \]
    and
    \[
    \frac{in}{\rho} = n+\frac{tn}{\rho} =\frac{\rho-t}{\rho}+
    \begin{cases}
        (\lambda i-1)p^j + M, & \text{if $j>0$;}\\
        (\lambda i-2), & \text{if $j=0$,}
    \end{cases}
    \]
    where
    \[
    M=(p-2)+ [j\geq 2] \cdot \sum_{k=1}^{j-1} (p-1)p^k.
    \]
    Note that adding the fractional parts of $n/\rho$ and $in/\rho$, we obtain
    \begin{equation}\label{addfractionals}
        \bigg\{ \frac{n}{\rho} \bigg\} + \bigg\{ \frac{in}{\rho} \bigg\} = \frac{2\rho -t-1}{\rho} =
        \begin{cases}
            1 + \frac{\rho-t-1}{\rho}, & \text{if $1\leq t <\rho$;} \\
            \frac{\rho-1}{\rho}, & \text{if $t=\rho$.}
        \end{cases}
    \end{equation}
    First, assume that $j=0$. By \eqref{addfractionals}, a carry with weight $v_p(U_\rho)$ occurs across the radix point if and only if $i<2\rho$, and we are left with counting carries produced in the parenthesized base-$p$ addition
    \[
    (\lambda i-2) + (\lambda-1) + [i<2\rho],
    \]
    that is, the carries in the addition of $(\lambda i-2)$ and $(\lambda-1)$, plus the number of carries in the addition of $(\lambda(i+1)-3)$ and $[i<2\rho]$. The latter addition produces exactly $v_p(\Lambda_{i,0})$ carries. Now, assume that $j\geq 1$. Then, by \eqref{addfractionals}, we have one carry across the radix point with weight $v_p(U_\rho)$ when $i<2\rho$, which generates, given the $p$-ary digits of $N$, at least $j$ other carries with total weight $j+\delta_j$. When $i=2\rho$, there is no carry across the radix point, but the addition of the first digits of $N$ and $M$ yields
    \[
    (p-1)+(p-2) = 
    \begin{cases}
        p+(p-3), & \text{if $p\geq 3$;} \\
        p-1, & \text{if $p=2$,}
    \end{cases}
    \]
    thus generating at least $j-1$ other carries when $p\geq 3$, but no carry when $p=2$. If $p=2$ and $j\geq 2$, then adding the second digits of the integer parts always produces one carry, which generates at least $j-2$ further carries. Therefore, whether $i<2\rho$, or $i=2\rho$ with $p\geq 3$ or $p=2$, we are left with counting carries in the parenthesized addition
    \[
    (\lambda i-1) p^j + (\lambda-1)p^j +
    \begin{cases}
        0, &\text{if $i=2\rho$, $p=2$ and $j=1$;}\\
        p^j, &\text{otherwise,}
    \end{cases}
    \]
    and we conclude the proof in the same way we did in Lemma \ref{valuationlucasnomials}.
\end{proof}

\section{First consequences}\label{Section5}

In this section, we give some consequences of Lemmas \ref{valuationlucasnomials} and \ref{valuationlucasnomials2}. The theorem below acts as an $l$-Lucasnomial-Catalan-number version of Theorem \ref{Thm1}. Most of the other results highlight infinite families of integers $n\geq 1$ such that $C_{U,l}(n)$ is, or is not, an integer.

\begin{theorem}\label{ThmRho>l}
    Let $p\nmid Q$ satisfy $\rho_U(p)>l$. Then $v_p(C_{U,l}(n))\geq 0$ for all $n\geq 1$.
\end{theorem}

\begin{proof}
    Let $p$ be such a prime. If $p\nmid U_{n+1}$, then it suffices to show that the $l$-Lucasnomial coefficient has non-negative valuation. Corollary \ref{TLucMultiVal} and \cite[($11$) and the comment below]{Ba1} ensure this is the case. Assume $p\mid U_{n+1}$ and write $n=\lambda \rho p^j-1$ for some $\lambda\geq 1$ with $p\nmid \lambda$. Since $\rho=\rho_U(p)>l$, by Lemma \ref{valuationlucasnomials}, Corollary \ref{TLucMultiVal}, and Proposition \ref{lawsAandR}, we have
    \[
    v_p\biggl(\binom{ (l+1)n}{n }_{U,l} \biggr) = \sum_{i=1}^{l} v_p\biggl( \binom{ (i+1)n}{n }_U \biggr) \geq l (j+\delta_j+ v_p(U_\rho)) = v_p(U_{n+1}^l),
    \]
    that is, $v_p(C_{U,l}(n))\geq 0$.
\end{proof}l

The integrality of Lucasnomial Catalan numbers, proved in a paper of Ekhad \cite{Ek}, is recovered for $l=1$. Indeed, $p\nmid Q$ and $\rho_U(p)\geq 2$ imply that $v_p(C_{U,1}(n))\geq 0$ for all $n\geq 1$ by Theorem \ref{ThmRho>l}. But the regularity of $U$ ensures that these primes are exactly those that divide $U$ at some term. Therefore, we conclude that $C_{U,1}(n)$ is always an integer.

\begin{lemma}\label{lemmap>l}
    If there exists a prime number $p>l$, $p\nmid Q$, of rank $l$ in $U$, then $C_{U,l}(lp^j-1)$ is not an integer for all $j\geq 0$.
\end{lemma}

\begin{proof}
    Since $p>l=\rho_U(p)$, we have $l\leq p-1$ and $p>2$. Putting $n=lp^j-1$, by Lemma \ref{valuationlucasnomials} and Corollary \ref{TLucMultiVal}, we obtain
    \begin{align*}
        v_p\biggl(\binom{ (l+1)n}{n }_{U,l}\biggr) &= \sum_{i=1}^{l} v_p\biggl( \binom{ (i+1)n}{n }_U \biggr) \\
        &= lj+(l-1)v_p(U_l)+[j>0]\cdot v_p(l)+ \sum_{i=1}^{l-1} v_p(i) \\
        &=lj+(l-1)v_p(U_l),
    \end{align*}
    where the last equality comes from $p>l\geq i$. Since the valuation of the multinomial coefficient is less than $v_p(U_{n+1}^l) = l(v_p(U_l)+j)$, we obtain $C_{U,l}(lp^j-1) \not \in \mathbb{Z}$.
\end{proof}

\begin{proposition}\label{Propp+1}
    Assume $p>3$ and $\rho_U(p)=p+1$. If there exists a prime $q\nmid Q$ of rank $p$ in $U$, then $C_{U,p+1}(pq^j-1)$ is not an integer for all $j\geq 0$.
\end{proposition}

\begin{proof}
     Put $n=pq^j-1$ and $l=p+1$. Since $\rho_U(q)=p$ and $\rho_U(p)=p+1$, we know that $p$ and $q$ are distinct primes. Moreover, $q>2$ since otherwise $p=\rho_U(q)$ would be equal to either $2$ or $3$, which contradicts the assumption $p>3$. By Lemmas \ref{valuationlucasnomials} and \ref{valuationlucasnomials2}, and Corollary \ref{TLucMultiVal}, we have
     \[
     v_q\biggl(\binom{ (l+1)n}{n }_{U,l} \biggr) = lj + (l-1)v_q(U_p)+ \sum_{i=1}^{p-1} v_q(i) +v_q(p+1-[j=0]),
     \]
     for every prime $q>2$. Since $p$ is the rank of $q$ in $U$, we have $p\mid q\pm 1$ by Proposition \ref{lawsAandR}, and we have $v_q(p)=0$ and $v_q(i)=0$ for all $1\leq i <p-1$. Next, we determine whether $q\mid p\pm 1$. Solving $p\mid q\pm1$ and $q\mid p\pm1$ for primes $p>3$ and $q> 2$ yields no solution. Hence
     \[
     v_q\biggl(\binom{ (l+1)n}{n }_{U,l}\biggr) = lj + (l-1)v_q(U_p),
     \]
     which, after comparing with $v_q(U_{n+1}^l) = l(j+v_q(U_p))$, proves the result.
\end{proof}

\begin{theorem}\label{thmDefect}
    Assume $C_{U,l}=(C_{U,l}(n))_{n\geq 1}$ is a multi-Catalan sequence, $l\geq 2$. Then 
    \begin{enumerate}[label=(\arabic*)]
        \item If $l-1\not \in \mathbb{P}$ or $l-1 \in \mathbb{P}$ of rank $\rho\ne l$ in $U$, then $U$ is $l$-defective,

        \item If $l-1 \in \mathbb{P}_{>3}$ of rank $l$ in $U$, then $U$ is $(l-1)$-defective.
    \end{enumerate}
\end{theorem}

\begin{proof}
    Assume $U$ is not $l$-defective. Then there exists a prime $p\nmid \Delta Q$ of rank $l$ in $U$ and by Lemma \ref{lemmap>l}, we must have $p\leq l$. Moreover, $p\nmid \Delta$ implies that $p\ne l$. However, by Proposition \ref{lawsAandR} we have $l\mid p\pm 1$, thus $l$ must be equal to $p+1$. The first point is done since we obtained a contradiction. For the second point, Proposition \ref{Propp+1} ensures that $U$ is $(l-1)$-defective.
\end{proof}

\begin{lemma}\label{lemmarho<2rho}
    If there exists an odd prime $p\nmid Q$ of rank $\rho$ in $U$ with $\rho\leq l \leq \min{(2\rho, p-1)}$, then $C_{U,l}(\rho p^j-1) $ is non-integral for all $j\geq 1$.
\end{lemma}

\begin{proof}
    Write $l=\rho+t$ with $0\leq t \leq  \rho$ and put $n=\rho p^j-1$. Using Lemmas \ref{valuationlucasnomials} and \ref{valuationlucasnomials2}, Corollary \ref{TLucMultiVal}, and since $p>2$ and $j> 0$, we obtain
    \begin{align*}
        v_p\biggl(\binom{(l+1)n}{n}_{U,l}\biggr) = lj+(l-1)v_p(U_\rho) + \sum_{ \subalign{i=1 \\ i\ne \rho} }^l v_p(i),
    \end{align*}
    when $l<2\rho$, and
    \[
    v_p\biggl(\binom{(l+1)n}{n}_{U,l}\biggr) = lj+(l-2)v_p(U_\rho)  + \sum_{ \subalign{i=1 \\ i\ne \rho} }^{l} v_p(i),
    \]
    when $l=2\rho$. Therefore, from the inequality $l<p$, we obtain
    \[
    v_p\biggl(\binom{(l+1)n}{n}_{U,l}\biggr) = lj +  (l-1-[l=2\rho])v_p(U_\rho),
    \]
    which is less than $v_p(U_{n+1}^l)=l(j+v_p(U_\rho))$.
\end{proof}

Until now, almost all results presented in this section are used to disprove the integrality of $C_{U,l}$ sequences. However, some exceptional pairs $(U,l)$ will later be found. Therefore, we state a final lemma to help us prove that $C_{U,l}$ is a multi-Catalan sequence for these exceptional cases.

\begin{lemma}\label{lemmap+1<l<}
    Let $n=\lambda \rho p^j-1$ for some $j\geq 0$ and $\lambda\geq 1$, $p\nmid \lambda$. Assume $\rho =p-\legendre{\Delta}{p}$, $v_p(U_\rho)=1$ and 
    \[
    \max(\rho,p) \leq l <2\rho.
    \]
    Then $v_p(C_{U,l}(n))\geq 0$, except for $n=\rho-1$ and $\rho=p=l$, or $\rho=p-1$ and $l=p$.
\end{lemma}

\begin{proof}
    Assume $j>0$. Then, by Lemmas \ref{valuationlucasnomials} and \ref{valuationlucasnomials2}, Corollary \ref{TLucMultiVal}, and the assumption $v_p(U_\rho)=1$, we have
    \[
    v_p\biggl(\binom{(l+1)n}{n}_{U,l}\biggr) \geq l(j+\delta_j)+(l-1) + \sum_{i=1}^l v_p(\Lambda_i),
    \]
    where $\Lambda_i = \lambda(i+1)-1$ for all $1\leq i \leq l$. Since $l\geq p$, there exists $1\leq i\leq l$ such that $\lambda(i+1)\equiv 1$ (mod $p$). Therefore,
    \[
    v_p\biggl(\binom{(l+1)n}{n}_{U,l}\biggr) \geq l(j+\delta_j+1) =v_p(U_{n+1}^l).
    \]
    Now, assume $j=0$. Then, using Lemmas \ref{valuationlucasnomials} and \ref{valuationlucasnomials2}, and Corollary \ref{TLucMultiVal}, we obtain
    \[
    v_p\biggl(\binom{(l+1)n}{n}_{U,l}\biggr) \geq (l-1) +\sum_{i=1}^{\rho-1} v_p(\Lambda_i) + [l>\rho] \cdot \sum_{ i=\rho+1 }^l v_p(\Lambda_{i,0}).
    \]
    Let $1\leq \mu \leq p-1$ be such that $\mu \lambda \equiv 1$ (mod $p$) and assume $\lambda \not \equiv 1$ (mod $p$). Thus $\mu\geq 2$, and taking $i=\mu-1 <\rho$ yields
    \[
    \lambda(i+1) \equiv 1 \pmod{p},
    \]
    and it follows that $v_p(C_{U,l}(\lambda \rho-1))\geq 0$. Next, assume that $p\mid \lambda-1$ and $\lambda\geq 2$. Let $s,\lambda_s\geq 1$ and $b\geq 0$ be integers such that $\lambda=1+\lambda_s p^s + bp^{s+1}$ and $\lambda_s<p$. Put
    \[
    i=\begin{cases}
        \floor{p/\lambda_s}, & \text{if $\lambda_s>1$;}\\
        p-1, & \text{if $\lambda_s=1$,}
    \end{cases}
    \]
    and let us count the number of carries produced in the base-$p$ addition $(\lambda-1) + (\lambda i-1)$. Since $1\leq i\leq p-1$, we have
    \[
    (i\lambda-1)= (i-1) + i\lambda_s p^s + dp^{s+1},
    \]
    for some integer $d\geq 0$ and with $i\lambda_s<p$. Adding the $s$-th digits of $(\lambda-1)$ and $(\lambda i-1)$ in base $p$ yields a carry since $\lambda_s+i\lambda_s=(i+1)\lambda_s \geq p$. By Lemmas \ref{carriesSum} and \ref{valuationlucasnomials},
    \[
    v_p\biggl(\binom{(l+1)n}{n}_{U,l}\biggr) \geq l = v_p(U_{n+1}^l).
    \]
    The only remaining case is $n=\rho-1$, for which we have
    \begin{equation}\label{eqVal}
        v_p\biggl(\binom{(l+1)n}{n}_{U,l}\biggr) = (l-1) + \sum_{i=1}^{\rho-1} v_p(i) + [l>\rho] \cdot \sum_{i=\rho+1}^{l} v_p(i-1).
    \end{equation}
    The $p$-adic valuation in \eqref{eqVal} is equal to $l-1$ when $\rho=p=l$, or $\rho=p-1$ and $l=\rho+1$, and is at least equal to $l=v_p(U_{n+1}^l)$ otherwise.
\end{proof}

\section{Multi-Catalan sequences}\label{Section6}

We are now ready to prove the main results on $l$-Lucasnomial Catalan numbers. First, we prove the integrality of all but a finite number of terms of $C_{U,l}$ when $(U,l)$ is an exceptional pair displayed in Table \ref{TableC}. Then, we proceed to show that for every other pair, there exists an infinite family of integers $n\geq 1$ such that $C_{U,l}(n)$ does not belong to $\mathbb{Z}$.

\begin{lemma}\label{exceptions}
    Assume $l\geq 2$ and let $(U,l)$ be a pair in Table \ref{TableC}. Then the numbers
    \[
    C_{U,l}(n) = \frac{1}{U_{n+1}^l} \binom{(l+1)n}{n}_{U,l},
    \]
    are integers for all $n\geq 1$, except for
    \[
    n =
    \begin{cases}
        p-1, & \text{if $p\mid \Delta$ and $l=p$;} \\
        p-2, & \text{if $\rho_U(p)=p-1$ and $l=p$;} \\
        5, & \text{if $U=U(1,-1)$ and $l=6$.}
    \end{cases}
    \]
    In particular, $C_{U,l}(n)$ is an integer for every $n\geq l$.
\end{lemma}

\begin{proof}
    By Theorem \ref{ThmRho>l}, $v_p(C_{U,l}(n))\geq 0$ for every prime $p\nmid Q$ with rank $\rho>l$ in $U$. For every pair $(U,l)$ in Table \ref{TableC}, except $(U(1,-1),6)$ and $(U(1,2),8)$, we look at the first terms of $U$, up to $U_l$, and check that all prime factors of those terms satisfy the hypotheses of Lemma \ref{lemmap+1<l<}. It follows that $v_p(C_{U,l}(n))\geq 0$ for every prime $p\nmid Q$ of rank $\rho\leq l$ in $U$ and every $n\geq 1$, except for the few exceptions stated in Lemma \ref{lemmap+1<l<}. Since $U$ is regular, $p\nmid Q$ is equivalent to $p\mid U_m$, for some $m\geq 1$. Hence, for all but finitely many $n\geq 1$, we have $v_p(C_{U,l}(n))\geq 0$ for all primes. Let $l=6$ and $U=U(1,-1)$. The above method works for every prime but $2$. Indeed, the first terms of $U$ up to $U_6$ are
    \[
    0,1,1,2,3,5, \text{ and } 8,
    \]
    and we see that $\rho =\rho_U(2)=3$ and $l=2\rho$, thus Lemma \ref{lemmap+1<l<} does not apply. Since there is nothing to prove when $2\nmid U_{n+1}$, assume that $2\mid U_{n+1}$ and write $n+1=\lambda \rho 2^j$ for some $j\geq 0$ and $\lambda\geq 1$, with $2\nmid \lambda$. By Lemmas \ref{valuationlucasnomials} and \ref{valuationlucasnomials2}, and Corollary \ref{TLucMultiVal}, we have
    \[
    v_2\biggl( \binom{7n}{n}_{U,6} \biggr) \geq (l-2) + v_2(3\lambda-1) + v_2(6\lambda-2) \geq l = v_p(U_{n+1}^l),
    \]
    when $j=0$, where we used $2\nmid \lambda$, $v_2(U_3)=1$, and $\delta_0 = 0$. When $j>0$, we have
    \begin{equation}\label{vall=6j>0}
        v_2\biggl( \binom{7n}{n}_{U,6}\biggr) = 5j+9+\sum_{i=1}^5v_2(\lambda(i+1)-1) + [j>1](j+v_2(7\lambda-1)-1) +\sum_{i=1}^6 c_i,
    \end{equation}
    where we used that $2\nmid \lambda$, $v_2(U_3)=1$, and $\delta_j = 1$ and where $c_i$ is the number of carries in the base-$p$ addition of $(\lambda-1)$ and $(\lambda i-1)$. When $j\geq 2$, we have
    \[
    v_2\biggl( \binom{7n}{n}_{U,6} \biggr) = 6j+8+v_2(3\lambda-1) +v_2(5\lambda-1)+v_2(7\lambda-1)+\sum_{i=1}^6 c_i,
    \]
    which is at least equal to $6(j+2)=v_2(U_{n+1}^6)$ since $2\nmid \lambda$ and we either have $v_2(3\lambda-1)\geq 2$ or $v_2(5\lambda-1)\geq 2$, depending on the value of $\lambda \bmod{4}$. When $j=1$ and $\lambda=1$, that is, when $n=5$, we find
    \[
    v_2\biggl( \binom{7n}{n}_{U,6} \biggr) = 17 < v_2(U_{6}^6)=18,
    \]
    so that $C_{U,6}(5) \not \in \mathbb{Z}$. Next, assume $j=1$ and $\lambda\geq 3$. We may write
    \[
    \lambda=1+2^s + b2^{s+1},
    \]
    for some integers $s\geq 1$ and $b\geq 0$. When $i=1$, adding $(\lambda-1)$ to $(\lambda i-1)$ produces at least one carry from the addition of the $s$-th digits. Hence
    \[
    v_2\biggl( \binom{7n}{n}_{U,6} \biggr) \geq 5j+10+v_2(3\lambda-1)+v_2(5\lambda-1) \geq 18 = v_2(U_6^6),
    \]
    by \eqref{vall=6j>0}. That makes the case $U=U(1,-1)$ and $l=6$ complete. Let $l=8$ and $U=U(1,2)$. Here, the method used at the beginning of the proof in the general case works for every prime but $3$. Indeed, the first terms up to $U_8$ are
    \[
    0,1,1,-1,-3,-1,5,7, \text{ and } -3,
    \]
    and we see that $\rho=\rho_U(3)=4$ and $l=2\rho$, thus Lemma \ref{lemmap+1<l<} does not apply. Again by Lemmas \ref{valuationlucasnomials} and \ref{valuationlucasnomials2}, and Corollary \ref{TLucMultiVal}, we have
    \[
    v_3\biggl( \binom{9n}{n}_{U,8} \biggr) \geq 8j+6 + [j>0]\cdot (v_3(\Lambda_4 ) + v_3(\Lambda_{8,j}))+ \sum_{i=1}^3 v_3(\Lambda_i) + \sum_{i=5}^7 v_3(\Lambda_{i,j}).
    \]
    Let $\mu\in \{1,2\}$ be such that $\mu \lambda \equiv 1$ (mod $3$) and put $i=1$ when $\mu=2$, and $i=3$ when $\mu=1$. In addition, put $i_1=i+3$ when $j>0$, and $i_1=i+4$ when $j=0$. Hence
    \[
    v_3(\Lambda_i)\geq 1 \quad \text{and} \quad v_3(\Lambda_{i_1,j})\geq 1,
    \]
    and we obtain $v_3\bigl(\binom {9n}{n}_{U,8}\bigr)\geq 8(j+1) = v_3(U_{n+1}^8)$ for all such $n\geq 1$.
\end{proof}

\begin{lemma}\label{l=6}
    The two sequences $U(1,-1)$ and $U(1,2)$ are the only $\Delta$-regular Lucas sequences such that the numbers
    \[
    C_{U,6}(n) = \frac{1}{U_{n+1}^6} \binom{7n}{n}_{U,6},
    \]
    are integers for all $n\geq 1$, except $n=5$ when $U=U(1,-1)$.
\end{lemma}

\begin{proof}
    Let $C_{U,6}$ be a multi-Catalan sequence and note that
    \begin{equation}\label{U6}
        U_6 = U_2U_3 (P^2-3Q).
    \end{equation}
    By Theorem \ref{thmDefect}, either $U$ is $6$-defective or, when $\rho_U(5)=6$, $U$ is $5$-defective. From \eqref{U6}, having $\rho_U(5)=6$ implies that $5\mid P^2-3Q$. Testing the seven $5$-defective sequences displayed in Table \ref{TableA}, we find that only $U(1,2)$ satisfies this divisibility relation. By Lemma \ref{exceptions}, we know that $C_{U,6}$ is multi-Catalan for $U(1,2)$. Therefore, any other $U$ such that $C_{U,6}$ is multi-Catalan must be $6$-defective. For the four families of $6$-defective sequences, we work on $U_3$ to find multi-Catalan sequences. First, by Lemma \ref{lemmarho<2rho}, $U_3$ may only have $2$, $3$ or $5$ as prime divisors and, since $\gcd{(U_2,U_3)}=1$ by regularity, these would be primitive prime divisors of rank $3$. Moreover, $5$ is actually not a divisor of $U_3$. Indeed, if it were the case, then we would have $\rho_U(5)=3$ and
    \[
    v_5\biggl( \binom{7( 5^j\rho-1)}{5^j\rho-1}_{U,6} \biggr) = 6j + 4v_5(U_3)+1 < 6(j + v_5(U_3)) = v_5(U_{5^j \rho}^6),
    \]
    for all $j\geq 0$, by Lemmas \ref{valuationlucasnomials} and \ref{valuationlucasnomials2}. Hence, $C_{U,6}( 5^j\rho-1)$ is non-integral for all $j\geq 0$, contradicting our hypothesis. Therefore, $U_3=\pm 2^a 3^b$ for some $a,b\geq 0$. If $b\geq1$ then $3$ has rank $3$ in $U$ and, by Lemmas \ref{valuationlucasnomials} and \ref{valuationlucasnomials2},
    \[
    v_3\biggl( \binom{7(3^{j+1}-1)}{3^{j+1}-1}_{U,6} \biggr) =6j+4b+2,
    \]
    for every $j\geq 1$, which is greater than or equal to $v_p(U_{3^{j+1}}^6) = 6(j+b)$ if and only if $b=1$. Hence $b\in \{0,1\}$. Let us now study the four families.
    \begin{enumerate}[label=(\arabic*), leftmargin = \widthof{[(1)]}]
        \item Let $Q=(P^2-1)/3$ with $3\nmid P\geq 4$. Note that $U_3= (2P^2+1)/3$ is an odd integer, so $a=0$. Hence either $U_3=\pm1$ or $U_3=\pm 3$, and it follows that $P^2 \in \{-5,-2,1,4\}$. But $P$ is an integer greater than or equal to $4$, so there is no solution.
        
        \item Let $Q=(P^2\pm3)/3$ with $3\mid P$. Note that from 
        \[
        U_3=\frac{2P^2\pm 3}{3} = \pm 2^a 3^b,
        \]
        we obtain $\pm3 = \pm 2^a 3^{b+1}-2P^2$. Since $2$ and $9$ do not divide $3$, $a$ and $b$ must be equal to zero. It follows that $P^2\in \{-3,0,3\}$, giving no solution since $P$ is a non-zero integer.

        \item Let $Q=(P^2-(-2)^i)/3$ with $P\equiv \pm1$ (mod $6$), $i\geq 1$ and $(P,i)\ne (1,1)$. Since $2\nmid PQ$, we see that $\rho_U(2)=3$. Furthermore, note that
        \[
        \delta = v_2\biggl( \frac{P^2-3Q}{2}\biggr) = v_2\biggl( \frac{(-2)^i}{2} \biggr)= i-1,
        \]
        and
        \begin{equation}\label{val1}
            v_2\biggl( \binom{7(3\cdot 2^j-1)}{3\cdot 2^j-1}_{U,6} \biggr) = 6j + 5\delta + 4a + 3,
        \end{equation}
        for every $j\geq 1$. For $6j+5\delta+4a+3$ in \eqref{val1} to be at least equal to $v_2(U_{3\cdot 2^j}^6)= 6(j+\delta+a)$, the inequality $\delta+2a \leq 3$ must be satisfied. Since $a\geq 1$, we have $a=1$ and $i\in \{1,2\}$ as the only solutions. If $i=1$, then
        \[
        U_3=\frac{2P^2-2}{3} = \pm 2 \cdot 3^b,
        \]
        i.e., $P^2=1\pm 3^{b+1}$. Thus, $P^2\in \{ -8,-2,4,10 \}$ and $P=\pm2$. However, $P$ is congruent to $\pm1$ modulo $6$, so there is no solution in the case $i=1$. If $i=2$, then, by the same method, $P^2=-2\pm 3^{b+1}$. Hence $P^2\in \{ -11,-5,1,7\}$ and we find the Fibonacci sequences $U(\pm 1,-1)$ as candidates. These are two of the exceptions in Lemma \ref{exceptions}.

        \item Let $Q=(P^2\pm 3\cdot 2^i)/3$ with $P\equiv 3$ (mod $6$) and $i\geq 1$. Then
        \[  
        U_3=\frac{2P^2\pm 3\cdot 2^i}{3} = \pm 2^a 3^b,
        \]
        and, using the previous method, we find $a=1$ and $i\in \{1,2\}$. But $9\mid P^2$, so that $b=0$. Solving for $P$ and $i$, we find that $P=\pm 3$ and $i=2$. Hence,
        \[
        \quad Q=\frac{P^2+3\cdot 2^i}{3} = 7.
        \]
        However, for the sequences $U(\pm3,7)$, we have $\rho_U(59)=5\leq 6 <\min{(2\cdot 5, 59)}=10$. By Lemma \ref{lemmarho<2rho}, $C_{U,6}$ is not a multi-Catalan sequence for $U(\pm3,7)$.
    \end{enumerate}
\end{proof}

\begin{theorem}\label{TheTHM}
    Assume $l\geq 2$. There are no $\Delta$-regular Lucas sequences $U$ with $C_{U,l}$ multi-Catalan other than the exceptions of Lemma \ref{exceptions}.
\end{theorem}

\begin{proof}
    If $l>30$, then by Theorem \ref{thmDefect} the sequence $C_{U,l}$ is multi-Catalan only if $U$ is either $l$ or $(l-1)$-defective. The primitive divisor theorem ensures that this is not the case for $l>31$. For $l=31$, $U$ must only be $l$-defective, thus the same conclusion holds.

    Let $5\leq l \leq 30$ be an integer not equal to $6$. Using both points of Theorem \ref{thmDefect} and inspecting Table \ref{TableA}, we find that the only possible sequences $U$ are those in line $l$ of Table \ref{TableA} if $l\in\{5, 7,10,12,13,18,30 \}$, in line $13$ if $l=14$ and in lines $7$ or $8$ if $l=8$. We start with $l=30$, where $U=U(1,2)$ is the only candidate. Note that the prime $271$ has rank $\rho=17$ in $U(1,2)$. Since $\rho \leq l=30 <\min{(2\rho,271)}$, Lemma \ref{lemmarho<2rho} ensures that $C_{U,30}$ is not multi-Catalan. The same method applies to every other case, except for the exceptions of Lemma \ref{exceptions}. Table \ref{TableB} lists some possible choices for $p$ and $\rho_U(p)$ in each case.
    
    We now study the remaining cases, that is, $l=2,3,4$ or $6$, but the case $l=6$ has already been studied in Lemma \ref{l=6}.
    \begin{enumerate}[label = (\arabic*), leftmargin = \widthof{[(1)]}]
        \item Let $l=4$. Assume $C_{U,4}$ is a multi-Catalan sequence. By Lemma \ref{lemmarho<2rho}, only the primes $2$ and $3$ may have rank lower than or equal to $4$ in $U$. Let $p\in \{2,3\}$ and assume $\rho=\rho_U(p)=2$. By Lemmas \ref{valuationlucasnomials} and \ref{valuationlucasnomials2}, and Corollary \ref{TLucMultiVal}, we have
        \[
        v_p\biggl(\binom{5n}{n}_{U,4}\biggr) = 4j +2v_p(P) +
        \begin{cases}
            1, & \text{if $p=3$;} \\
            3, & \text{if $p=2$,}
        \end{cases}
        \]
        for every $n\geq 1$ of the form $n=\rho p^j-1$, $j\geq 0$. Note that $\delta_j$ does not appear here since $\delta_j=0$ for $j\geq 0$ when $\rho_U(2)=2$. Since $v_p(U_{n+1}^4) = 4j+4v_p(P)$, we must have $P=\pm1$ or $P=\pm 2$, otherwise $C_{U,4}(\rho p^j-1)$ is not integral for all $j\geq 0$. If $3$ has rank $\rho=4$, then
        \[
        v_3\biggl(\binom{5n}{n}_{U,4}\biggr) = 4j +3v_p(U_4) +1,
        \]
        for every $n\geq 1$ of the form $n=3^j\rho-1$, $j\geq 0$. Since $v_3(U_{n+1}^4)=4j+4v_p(U_4)$ and $U_4=U_2(P^2-2Q)$, it follows that $P^2-2Q = \pm1 $ or $\pm 3$. Hence
        \[
        Q= \frac{P^2\pm1}{2} \quad \text{or} \quad Q=\frac{P^2\pm3}{2},
        \]
        and in particular $P=\pm1$. Therefore, $(P,Q) \in \{ (\pm1,0), (\pm1,1), (\pm1,2), (\pm1,-1) \}$. Since $Q\ne0$ and $U(\pm1,1)$ is a degenerate Lucas sequence with $U_3=0$, we have $(P,Q) =(\pm1,2)$ or $(\pm1,-1)$, which are exceptions found in Lemma \ref{exceptions}.

        \item Let $l=3$. Following the method we used for $l=4$, we find that the only primes susceptible of having rank $2$ or $3$ are the primes $2$ and $3$. By Lemmas \ref{valuationlucasnomials} and \ref{valuationlucasnomials2}, Corollary \ref{TLucMultiVal} and the assumption that $p\in \{2,3\}$ has rank $2$ in $U$, we have
        \[
        v_p\biggl(\binom{4n}{n}_{U,3}\biggr) = 3j +2v_p(P) +1,
        \]
        for every $n\geq 1$ of the form $n=\rho p^j-1$, $j\geq 0$. Assuming $\rho_U(p)=3$, we obtain
        \[
        v_p\biggl(\binom{4n}{n}_{U,3} \biggr) = 3(j+\delta_j) +2v_p(U_3) +1,
        \]
        for every $n\geq 1$ of the form $n=\rho p^j-1$, with $j\geq 0$. Therefore, $U_2$ and $U_3$ are coprime divisors of $6$, positive or negative, where the coprimality is a consequence of the regularity of $U$. The pair $(U_2,U_3)$ is then one of the following:
        \begin{align*}
            &(1,\pm1), (1,\pm2), (1,\pm3), (1, \pm6), \\
            &(2,\pm1), (2,\pm3), (3,\pm1), (3, \pm2), (6,\pm1),
        \end{align*}
        assuming $P>0$. Since $U_3=P^2-Q$, we find $(P,Q)$ to be equal to either $(1,0)$, $(2,1)$, or one of the pairs displayed in line $5$ of Table \ref{TableC}. However, $Q\ne 0$ and $U(2,1)$ is not $\Delta$-regular. Hence $(P,Q)$ is in line $5$ of Table \ref{TableC}, that is, $U(P,Q)$ is an exception mentioned in Lemma \ref{exceptions}.

        \item Let $l=2$. By the same method, only $2$ may have rank $2$ in $U$. In that case, putting $n=\rho 2^j-1$, we obtain
        \[
        v_2\biggl(\binom{3n}{n}_{U,2}\biggr) = 2j +v_p(P) +1,
        \]
        by Lemmas \ref{valuationlucasnomials} and \ref{valuationlucasnomials2}, and Corollary \ref{TLucMultiVal}. It follows that $(P,Q)$ is either equal to $(\pm2,Q)$ or to $(\pm1,Q)$, with $Q\ne 1$, since $(\pm1,1)$ is a degenerate Lucas sequence. Hence $U(P,Q)$ is an exception mentioned in Lemma \ref{exceptions}.
    \end{enumerate}
\end{proof}

Although $C_{U,l}$ has infinitely many non-integer terms for many pairs $(U,l)$, one can show that the set of $n\geq 1$ for which $C_{U,2}(n)$ is an integer has asymptotic density $1$. This can potentially be done by following the method used by Ballot \cite[Lemma 36]{Ba1}. We state this as a conjecture which we leave as an open problem.

\begin{conjecture}
    Let $l\geq 1$ be an integer and $U$ a $\Delta$-regular Lucas sequence. Then the numbers $C_{U,l}(n)$ are integers for almost all $n\geq 1$.
\end{conjecture}

However, some adjustments need to be made in the case $l=2$ to the proof given by Ballot. These adjustments can be made to prove the conjecture for small values of $l$, but we think that a different method is required to prove the general case.

\section{Appendix}\label{Section7}

\renewcommand{\arraystretch}{1.4}

In this appendix section, we display tables that are used in various proofs. The first table below lists $n$-defective $\Delta$-regular Lucas sequences $U(P,Q)$, with $P>0$ and $Q\ne 0$, for $n\geq 2$. It was put together by Ballot \cite{Ba1}, but was first given in different tables \cite{Ab, Bi}, where the parametrization was done in terms of $P$ and $\Delta$ instead. Note that the parameter $i$ is an integer.

\begin{table}[H]
    \centering
    \begin{tabular}{|c|c|}
         \hline
         $n$ & $(P,Q)$ with $P>0$ and $(P,Q)\ne (2,1)$\\
         \hline \hline
         $2$ & $(1,Q)$ with $Q\ne 1$; \quad $(2^i, Q)$ with $Q$ odd, $i\geq 1$ and $(Q,i)\ne (1,1)$ \\
         \hline
         $3$ & $(P, P^2\pm 1)$; \quad $(P,P^2\pm 3^i)$ with $3\nmid P$ and $i\geq 0$ \\
         \hline
         $4$ & $(P, \textstyle (P^2\pm 1)/2)$ with $P>1$ odd; \quad $(2i, 2i^2\pm1)$ with $i\geq 1$ \\
         \hline
         $5$ & $(1,-1)$ \quad $(1,2)$ \quad $(1,3)$ \quad $(1,4)$ \quad $(2,11)$ \quad $(12,55)$ \quad $(12,377)$ \\
         \hline
             & $(P,(P^2-1)/3)$ with $3\nmid P \geq 4$; \quad $(P,(P^2\pm3)/3)$ with $3\mid P$; \\
         
         $6$ & $(P,(P^2-(-2)^i)/3)$ with $P\equiv \pm 1$ (mod $6$), $i\geq 1$ and $(P,i) \ne (1,1)$; \\
         
             & $(P, (P^2\pm3\cdot 2^i)/3)$ with $P\equiv 3$ (mod $6$) and $i\geq 1$ \\
         \hline
         $7$ & $(1,2)$ \quad $(1,5)$ \\
         \hline
         $8$ & $(1,2)$ \quad $(2,7)$ \\
         \hline
         $10$ & $(2,3)$ \quad $(5,7)$ \quad $(5,18)$ \\
         \hline
         $12$ & $(1,-1)$ \quad $(1,2)$ \quad $(1,3)$ \quad $(1,4)$ \quad $(1,5)$ \quad $(2,15)$ \\
         \hline
         $13$ & $(1,2)$ \\
         \hline
         $18$ & $(1,2)$ \\
         \hline
         $30$ & $(1,2)$ \\
         \hline
    \end{tabular}
    \caption{List of all $n$-defective and $\Delta$-regular Lucas sequences, $n\geq 2$.}
    \label{TableA}
\end{table}

Let $l\geq 1$ be an integer and $U=U(P,Q)$, $P>0$, a Lucas sequence. Given a pair $(U,l)$, our second table provides possible choices for a prime $p\nmid Q$ and its rank $\rho_U(p)$ that validates the proof of Theorem \ref{TheTHM} through the use of Lemma \ref{lemmarho<2rho}. For instance, when $U=U(1,4)$ and $l=5$, we find that $(p,\rho_U(p))=(7,4)$. By Lemma \ref{lemmarho<2rho}, since $\rho_U(7)\leq 5\leq \min{(2\rho_U(7),6)} = 6$, the numbers $C_{U,5}(4\cdot 7^j-1)$ are not integers for all $j\geq 1$.

Note that the first column displays pairs $(P,Q)$ for Lucas sequences $U=U(P,Q)$, with $P>0$, and that the first line displays values of $l$. An empty cell means that it is not needed to apply Lemma \ref{lemmarho<2rho} for the corresponding pair $(U,l)$, thus no prime is provided.

\begin{table}[H]
    \centering
    \begin{tabular}{|c||c|c|c|c|c|c|c|c|c|}
         \hline
         $U\setminus l$ & $5$ & $7$ & $8$ & $10$ & $12$ & $13$ & $14$ & $18$ \\
         \hline \hline
         $(1,2)$ & & & & & $(17,9)$ & $(17,9)$ & $(17,9)$ & $(31,16)$ \\
         \hline
         $(1,5)$ & & $(11,5)$ & $(11,5)$ & & $(71,9) $ & & &   \\
         \hline
         $(1,4)$ & $(7,4)$ &  & & & $(43,11)$ & & &   \\
         \hline
         $(1,3)$ & &  & & & $(23,11)$ & & &   \\
         \hline
         $(1,-1)$ & &  & & & $(89,11)$ & & &   \\
         \hline
         $(2,15)$ & &  & & & $(43,11)$ & & &  \\
         \hline
         $(2,3)$ &  & & & $(73,9)$ & & & &  \\
         \hline
         $(5,18)$ & &  & & $(31,8)$ & & & &  \\
         \hline
         $(5,7)$ & &  & & $(19,9)$ & & & &  \\
         \hline
         $(2,7)$ & &  & $(17,6)$ & & & & & \\
         \hline
         $(2,11)$ & $(7,3)$ &  & & & & & & \\
         \hline
         $(2,55)$ & $(17,3)$ & & & & & & & \\
         \hline
         $(2,277)$ & $(31,3)$ & & & & & & & \\
         \hline
    \end{tabular}
    \caption{Choices of $(p,\rho_U(p))$ for a successful application of Lemma \ref{lemmarho<2rho}, $U=U(P,Q)$, $P>0$.}
    \label{TableB}
\end{table}

Our last table lists all pairs $(U,l)$, with $l\geq 2$ and $U$ a $\Delta$-regular Lucas sequence, such that the numbers
\[
C_{U,l}(n) = \frac{1}{U_{n+1}^l} \binom{(l+1)n}{n,\ldots,n}_{U}
\]
are integers for all but finitely many $n\geq 1$.

\begin{table}[H]
    \centering
    \begin{tabular}{|c|c|}
         \hline
         $l$ & $(P,Q)$ with $P>0$\\
         \hline \hline
         $2$ & $(1,Q)$ with $Q\ne 1$; \quad $(2, Q)$ with $Q\ne 1$ odd \\
         \hline
             & $(1,-5)$ \quad $(1,-2)$ \quad $(1,-1)$  \quad $(1,2)$ \quad  $(1,3)$ \quad  $(1,4)$  \\
         $3$ & $(1,7)$ \quad $(2,3)$ \quad $(2,5)$ \quad $(2,7)$ \quad $(3,7)$ \quad $(3,8)$  \\
             & $(3,10)$ \quad $(3,11)$ \quad $(6,35)$ \quad $(6,37)$ \\
         \hline
         $4$ & $(1,-1)$ \quad $(1,2)$ \\
         \hline
         $5$ & $(1,-1)$ \quad $(1,2)$ \quad $(1,3)$ \\
         \hline
         $6$ &  $(1,-1)$ \quad $(1,2)$ \\
         \hline
         $7$ & $(1,2)$  \\
         \hline
         $8$ & $(1,2)$ \\
         \hline
    \end{tabular}
    \caption{List of all multi-Catalan $\Delta$-regular Lucas sequences, $l\geq 2$.}
    \label{TableC}
\end{table}

\section{Acknowledgments}

We thank the referee for his or her positive report. We are grateful to Christian Ballot for his mathematical advice and writing commentaries.


\begin{thebibliography}{9}

\bibitem{Ab} M.\ Abouzaid, Les nombres de Lucas et Lehmer sans diviseur primitif, {\it J.\ Th\'eor.\ Nombres Bordeaux} {\bf 18} (2006), 299--313.

\bibitem{Ba1} C.\ Ballot, Lucasnomial Fuss-Catalan numbers and related divisibility questions, {\it J.\ Integer Sequences} {\bf 21} (2018),
\href{https://cs.uwaterloo.ca/journals/JIS/VOL21/Ballot/ballot30.pdf}{Article 18.6.5}.


\bibitem{Ba2} C.\ Ballot, Divisibility of Fibonomials and Lucasnomials via a general Kummer rule, {\it Fibonacci Quart.} {\bf 53} (2015), 194--205.


\bibitem{Ba3} C.\ Ballot, Divisibility of the middle Lucasnomial
coefficient, {\it Fibonacci Quart.} {\bf 55} (2017), 297--308.


\bibitem{Bi} Y.\ Bilu, G.\ Hanrot, and P.\ M.\ Voutier, Existence of primitive
divisors of Lucas and Lehmer numbers. With an appendix by M.\ Mignotte,
{\it J.\ Reine Angew.\ Math.} {\bf 539} (2001), 75--122.


\bibitem{Ek} S.\ Ekhad, The Sagan-Savage Lucas-Catalan polynomials have positive coefficients, 	arxiv preprint arXiv:1101.4060 [math.CO], 2011. Available at \href{http://arxiv.org/abs/1101.4060}{http://arxiv.org/abs/1101.4060}. 


\bibitem{KW} D.\ Knuth and H.\ Wilf, The power of a prime that divides a
generalized binomial coefficient, {\it J.\ Reine Angew.\ Math.} {\bf
396} (1989), 212--219.


\bibitem{Ku} E.\ Kummer, \"Uber die Erg\"anzungss\"atze zu den allgemeinen Reciprocit\"atsgesetzen, {\it J.\ Reine Angew.\ Math.} {\bf 44} (1852), 93--146.



\bibitem{Wi} H.\ C.\ Williams, \emph{\'Edouard Lucas and Primality
Testing}, Wiley, 1998.
\end{thebibliography}
\end{document}